\documentclass[12pt]{article}
\usepackage{amsfonts}
\usepackage{amssymb, amsmath, amsthm, eucal}
\usepackage{amsmath}
\usepackage{amssymb}
\usepackage{graphicx}%
\providecommand{\U}[1]{\protect\rule{.1in}{.1in}}
\newtheorem{thm}{Theorem}

\newtheorem{cor}{Corollary}
\newtheorem{ex}{Example}
\newtheorem{prop}{Proposition}
\newtheorem{rem}{Remark}
\theoremstyle{definition}

\usepackage{enumerate}
\begin{document}

\title{Normal harmonic mappings\thanks{The authors were partially supported by Fondecyt Grant  \#1150284.}}
\author{Hugo Arbel\'{a}ez\thanks{The first author was partially supported by the Universidad Nacional de Colombia, Hermes code 34044.}\\  Universidad Nacional de Colombia, sede Medell\'{\i}n, Colombia
\\
 Rodrigo Hern\'{a}ndez \\ Universidad Adolfo Iba\~nez, Vi\~na del Mar, Chile\\
 \and Willy Sierra\thanks{\noindent The author wishes to thank the Universidad del
Cauca for providing time for this work through research project VRI ID 4340.}\\Universidad del Cauca, Popay\'{a}n, Colombia }
\date{}
\maketitle

\begin{abstract}
The main purpose of this paper is to study the concept of normal function in the context of harmonic mappings from the unit disk $\mathbb{D}$ to the complex plane. In particular, we obtain necessary conditions for that a function $f$ to be normal.

\medskip

\medskip


\end{abstract}
 
\section{Introduction}

Let $f$ be a meromorphic function defined in the unit disk $\mathbb {D}$. It is known that $f$ is normal if the family $\mathcal{F}=\left\lbrace f\circ \sigma: \sigma \in \text {Aut}(\mathbb{D})  \right\rbrace $ is a normal family in the sense of Montel, where Aut($\mathbb{D}$) denotes the group of conformal automorphisms of $\mathbb{D}.$ This definition is due to Lehto and Virtanen \cite{LV}; however the concept of normal function was first studied by Yosida \cite{Yo}  and subsequently
by Noshiro \cite{N38} in a more general context. An important result in \cite{N38}, which allows to characterize the normal meromorphic functions, establishes that $f$ is normal if and only if
\begin{equation} \label{DI1}
\sup_{z\in \mathbb{D}}(1-|z|^2)f^{\#} (z)<\infty,
\end{equation}
where
\[f^{\#} (z)=\frac{|f'(z)|}{1+|f(z)|^2}
\]
is the spherical derivative of $f.$ The condition (\ref{DI1}) is equivalent to saying that $f$ is Lipschitz when regarded as a function from the hyperbolic disk $\mathbb{D}$ into the extended complex plane endowed with the chordal distance (see \cite{LV}). The univalent meromorphic functions and analytic functions which omit two values are some important examples of normal meromorphic functions.

Mainly because of its importance in geometric functions theory, specifically in the study of the behaviour in the boundary of meromorphic functions in the unit disk, many authors have  investigated properties of normal meromorphic functions, both, from the geometric point of view and from the analytic point of view; see for example \cite{ACP}, \cite{Po75}, and \cite{Y}. In particular, Ch. Pommerenke \cite{Po75} proved that a meromorphic function in $\mathbb{D},$ which satisfies the condition
\begin{equation} \label{DI1Po} 
\iint_\mathbb{D} {\left\lbrace f^{\#}(z)\right\rbrace }^2dA< \infty,
\end{equation}
is normal. 

In \cite{L} Lappan also considers real valued harmonic functions defined in $\mathbb{D}$, and he established that $u$ is normal if  
\begin{equation}\label{DI2}
\sup_{z\in \mathbb{D}}(1-|z|^2)\frac{|\text{grad} \, u(z)|}{1+u^2(z)}<\infty,
\end{equation} 
where \text{grad}$u$ is the gradient vector of $u$. In that paper, he shows that if $u$ is a harmonic normal function, and if $f=u+iv$ is analytic in $\mathbb{D}$, then $f$ is a normal function. In \cite{AL} the authors also prove geometric properties of real valued harmonic normal functions, as for example, a harmonic version of the Pommerenke's result cited above. Namely, a real valued harmonic function $u$ with the property
\[\iint_\mathbb{D} {\left\lbrace \frac{|\text{grad} \, u(z)|}{1+u^2(z)} \right\rbrace }^2dA< \infty,
\]
is normal.

Since the topic of harmonic mappings of complex value is one of the most studied in the area of complex analysis nowadays (see, for example, \cite{Harmonic S-1}, \cite{Bloch type}, \cite{John Ponnusamy}, \cite{SM90}); it seems natural to address the issue of normal harmonic mappings (of complex value) defined in $\mathbb{D}$. We remark that an important subject related with normal meromorphic functions is the concept of Bloch function, which has been studied by Colonna \cite{Co89} in the setting of harmonic mappings. See also \cite{Bloch type}. Along the paper we will consider harmonic mappings from $\mathbb{D}$ into $\mathbb{C},$ which have a canonical representation of the form $f=h+\overline{g},$ where $h$ and $g$ are analytic functions in $\mathbb{D}$; this representation is unique with the condition $g(0)=0.$ It is a classical result of Lewy \cite{Le} that a harmonic mapping is locally univalent in a domain $\Omega$ if and only if its Jacobian does not vanish. In terms of the canonical decomposition the Jacobian is given by $J_f=|h'|^2-|g'|^2$, and thus, a locally univalent harmonic mapping in a simply connected domain $D$ will be sense-preserving if $|h'|>|g'|$. The family of all sense-preserving univalent harmonic mappings defined in $\mathbb{D}$, normalized by $h(0) = 0,$ $g(0) = 0,$ and $h'(0) = 1,$ will be denoted by $S_H.$ Also, we denote by $S^0_H$ the subclass of functions in $S_H$ that satisfy the further normalization $g'(0) = 0$.

Following the above ideas, particularly the definition given by Colonna \cite{Co89} of Bloch harmonic function, we will say that a harmonic mapping $f$ defined in $\mathbb{D}$ is normal if it satisfies a Lipschitz type condition. In Section 2, analogously to (\ref{DI1}) and (\ref{DI2}), we prove that $f=h+\overline{g}$ is normal if and only if
\begin{equation} \label{DI3}
\sup_{z\in\mathbb{D}}\left(  1-\left\vert
z\right\vert ^{2}\right)  \frac{|h'(z)|+|g'(z)|}{1+\left\vert f\left(
	z\right)  \right\vert ^{2}}<\infty.
\end{equation}
We observe that if $f$ is analytic, equation (\ref{DI3}) reduces to equation (\ref{DI1}). In addition, we show some examples of functions which turn out to be normal and obtain several properties that normal harmonic mappings satisfy. In Section 3, we establish the main results of this work; in particular, we show that in the case of complex valued harmonic mappings we require an additional condition for a univalent function to be normal which cannot be omitted.

\section{Definition and properties}

We recall that the chordal distance on the extended complex plane $\widehat{\mathbb{C}}$ is defined by
\begin{equation*}
\chi(z,w)=\frac{|z-w|}{\sqrt{1+|z|^{2}}\sqrt{1+|w|^{2}}}; \qquad z,w\in\mathbb{C},
\end{equation*} 
and $\chi(z,\infty)=\left( 1+|z|^{2}\right)^{-1/2}$. If $P_z$ and $P_w$ are the points on the Riemann sphere, under stereographic projection, corresponding to $z$ and $w$ respectively, we have $\chi(z,w)=|P_z-P_w|.$ Therefore, $\chi(z,w)\leq\sigma(z,w)\leq L(\Gamma),$ where $\sigma(z,w)$ is the spherical distance of $z$ to $w,$ $\Gamma$ is any rectifiable curve in $\mathbb{C}$ with endpoints $z,w,$ and
\[L(\Gamma)=\int_{\Gamma}\frac{|d\xi|}{1+|\xi|^2}\]
is the spherical length of $\Gamma.$ Throughout the paper, given $z,w\in\mathbb{D},$ $\rho\left(  z,w\right)$ denotes the hyperbolic distance between $z$ and $w.$ So, if $\gamma$ denotes the hyperbolic geodesic joining $z$ to $w,$ then
\[\rho(z,w)=\int_{\gamma}\frac{|d\zeta|}{1-|\zeta|^2}.\,\]
More explicitly
\[\rho(z,w)=\frac{1}{2}\log\frac{1+r}{1-r}\,,\]
where
\[r=\left| \frac{z-w}{1-\bar{z}w}\right| .\]
With this notation, a harmonic mapping $f:\mathbb{D}\to\mathbb{C}$ is called a normal harmonic mapping, if 
\begin{equation*}
\sup_{z\neq w}\frac{\chi\left(  f\left(  z\right)  ,f\left(  w\right)
	\right)  }{\rho\left(  z,w\right)  }<\infty.
\end{equation*}
The following proposition provides an alternative method for determining when a harmonic mapping is normal. 
\begin{prop}\label{B}
	Let $f=h+\overline{g}$ be a harmonic mapping in $\mathbb{D}$. Then, $f$ is normal if and only if 
	\begin{equation} \label{C}
	\left\Vert f\right\Vert _{n}:=\sup_{z\in\mathbb{D}}\left(  1-\left\vert
	z\right\vert ^{2}\right)  \frac{|h'(z)|+|g'(z)|}{1+\left\vert f\left(
		z\right)  \right\vert ^{2}}<\infty.%
	\end{equation}
\end{prop}
\begin{proof}
	Suppose that $\left\Vert f\right\Vert_{n}<\infty$ and let $z,w \in\mathbb{D}.$ Thus, if $\gamma:\left[  0,1\right]  \rightarrow\mathbb{D}$ is the hyperbolic geodesic with endpoints $z$ and $w,$
	\[\chi\left(  f\left(  z\right)  ,f\left(  w\right)  \right) \leq\int_{f\circ\gamma}\frac{\left\vert d\xi\right\vert }{1+\left\vert\xi\right\vert ^{2}}=\int_{0}^{1}\frac{\left\vert df\left(  \gamma(t)\right)\gamma^{\prime}(t)\right\vert}{1+\left\vert f\left(  \gamma(t)\right)\right\vert ^{2}}\,dt,\]
	where $df$ stands for the differential of $f$. From here and \eqref{C}, we have
	\[\chi\left(  f\left(  z\right)  ,f\left(  w\right)  \right)\leq \left\Vert f\right\Vert _{n}\int_{0}^{1}\frac{\left\vert \gamma
		^{\prime}\left(  t\right)  \right\vert dt}{1-\left\vert \gamma\left(
		t\right)  \right\vert ^{2}}=\left\Vert f\right\Vert _{n}\rho\left(
	z,w\right).\]
	The proof that $f$ normal implies condition (\ref{C}) follows as in	\cite{Co89}, p. 831.
\end{proof}
\begin{rem}\label{rem 1}
	The proof really shows that
	\[\sup_{z\neq w}\frac{\chi\left(  f\left(  z\right)  ,f\left(  w\right)
		\right)  }{\rho\left(  z,w\right)  }=\sup_{z\in\mathbb{D}}\left(  1-\left\vert
	z\right\vert ^{2}\right)  \frac{|h'(z)|+|g'(z)|}{1+\left\vert f\left(
		z\right)  \right\vert ^{2}}.\]
	We note further that if $\varphi:\mathbb{D}\to\mathbb{D}$ is analytic and locally univalent it follows, by Schwarz-Pick's Lemma, that $\left\Vert f\circ\varphi \right\Vert_{n}\leq\left\Vert f\right\Vert_{n},$ with equality if $\varphi\in	\mathrm{Aut}\left(\mathbb{D}\right).$ Therefore, $f$ normal implies that $f\circ\varphi$ must also be normal. Also, a straightforward calculation shows that $\left\| f\right\| _{n}=\left\| 1/f\right\| _{n}.$
\end{rem}
Let  $\mathcal{F}$ be a family of sense-preserving harmonic mappings $f=h+\overline{g}$ in $\mathbb{D}$. The family $\mathcal{F}$ is said to be affine and linearly invariant if it closed under the two operations of Koebe transform and affine changes:
\[K_f(z)=\frac{f\left(\frac{z+\xi}{1+\overline{\xi} z} \right)-f(\xi)}{(1-|\xi|^2)h'(\xi)}, \,\,\,\,\,\,\,\,\, |\xi|<1,
\]
and
\[A_f(z)=\frac{f(z)+\epsilon \overline{f(z)}}{1+\epsilon g'(0)}, \,\,\,\,\,\,\,\,\, |\epsilon|<1.
\]
In \cite{SM90}, Sheil-Small offers an in-depth study of affne and linearly invariant families $\mathcal{F}$ of harmonic mappings in $\mathbb{D}$.
\begin{cor}
	The class of all normal harmonic mappings is an affine and linearly invariant family.
\end{cor}
\begin{proof}
	It is sufficient to prove that $A \circ f $ is normal for all affine harmonic mapping $A(z)=az+b\bar{z}$ with $|a|\neq|b|.$ This follows immediately from the equalities
	\[(A \circ f)_{z}=af_{z}+b\overline{f_{\overline z}}\qquad\text{and}\qquad (A \circ f)_{\overline {z}}=af_{\overline {z}}+b\overline{f_{z}}\,,\]
	which imply
	\begin{align*}
	(1-|z|^{2})\frac{|(A \circ f)_{z}|+|(A \circ f)_{\overline {z}}|}{1+|af+b\overline{f}|^{2}}
	&=(1-|z|^{2})\frac{|af_{z}+b\overline{f_{\overline z}}|+|af_{\overline {z}}+b\overline{f_{z}}|}{1+|af+b\overline{f}|^{2}}\\
	& = (1-|z|^{2})\frac{(|a|+|b|)(|f_{z}|+|f_{\overline z}|)}{1+\left| |a|-|b|\right|^{2}|f|^{2} }\\ 
	& \leq \left\| f\right\| _{n} \frac{(|a|+|b|)(1+|f|^{2})}{1+\left| |a|-|b|\right|^{2}|f|^{2}}.
	\end{align*}
	Hence $\left\|A\circ f\right\| _{n}<\infty.$
\end{proof}
Next, we shall give two examples to illustrate the above proposition. The examples, which are functions in $S^0_H,$ show that, unlike the meromorphic case, a univalent harmonic mapping is not necessarily normal.

In this example and subsequently $a \lesssim b$ means that $a\leq Mb$, for some constant $M>0$ independent on $a$ and $b$.

\begin{ex}\label{ex 1}
	We consider the function $L(z)=Re\left\lbrace \dfrac{z}{1-z} \right\rbrace +iIm \left\lbrace  \dfrac{z}{(1-z)^2}\right\rbrace,$ which we can write in the form $L=h+\overline{g},$ where
	\[h(z)=\frac{1}{2}\dfrac{z(2-z)}{(1-z)^{2}}\qquad\text{and}\qquad g(z) = -\frac{1}{2}\dfrac{z^{2}}{(1-z)^{2}}.\]
	We observe that
	\[h'(z)=\dfrac{1}{(1-z)^{3}}\qquad\text{ and }\qquad |g'|<|h'|,\]
	the last inequality being a consequence of the fact that $L$ is a sense-preserving harmonic mapping. It follows that
	\begin{align*}
	\frac{(1-|z|^{2})|h'(z)|}{1+|L(z)|^{2}}&= \frac{(1-|z|^{2})\left| \frac{1}{1-z}\right| ^{3}}{1+\frac{1}{4}\left| \dfrac{z(2-z)}{(1-z)^{2}}-\overline{\left( \frac{z}{1-z}\right) ^{2}}\right|^{2} }\\[2mm]
	&=\frac{(1-|z|^{2})\left| \frac{1}{1-z}\right| ^{3}}{1+\frac{1}{4}\left|-1+ \left( \dfrac{1}{1-z}\right) ^{2}-\overline{\left( \frac{z}{1-z}\right) ^{2}}\right|^{2} }\,,
	\end{align*}
	whence by defining $u=\dfrac{1+z}{1-z}=x+ iy,$ one sees that $x>0$ and
	\begin{align*}
	\frac{(1-|z|^{2})|h'(z)|}{1+|L(z)|^{2}}&=\frac{1}{2} \frac{|u+1|Re\left\lbrace  u\right\rbrace }{1+\frac{1}{4}|-1+iRe\left\lbrace  u\right\rbrace Im\left\lbrace  u\right\rbrace+Re\left\lbrace  u\right\rbrace| ^{2}}\\
	& \lesssim \frac{x+x^{2}+x|y|}{1 + \frac{1}{4} \left[ (x-1)^{2}+(xy)^{2}\right]}\\
	&:=J(x,y).
	\end{align*}
	We consider two cases. First, if $x|y|\leq 1,$ then
	\[J(x,y)\lesssim \dfrac{x^{2}+x+1}{1+\frac {1}{4} (x-1)^{2}}<K_{1},\]
	for all $x>0.$ In the other case, if $x|y|> 1,$
	\[J(x,y)\lesssim \dfrac{x(x+1)}{1+\frac {1}{4} (x-1)^{2}}+\frac{x|y|}{\frac{1}{4}(x|y|)^{2}}<K_{2}.\]
	Thus, from
	\[\left(  1-\left\vert
	z\right\vert ^{2}\right)  \frac{|h'(z)|+|g'(z)|}{1+\left\vert L\left(
		z\right)  \right\vert ^{2}}\leq \left(  1-\left\vert
	z\right\vert ^{2}\right)  \frac{2\left|h'(z) \right|  }{1+\left\vert L\left(
		z\right)  \right\vert ^{2}},\]
we conclude that $L$ is normal.
\end{ex}
\begin{ex}\label{ex 2}
	The harmonic map $f$ defined by 
	\begin{equation}\label{E}
	f(z)=\frac{z}{1-z}-\overline{\frac{z}{1-z}-\log (1-z),} 
	\end{equation}
	for all $z\in \mathbb{D}$, is not normal. Indeed, it is easy to check that
	\begin{equation*}
	\left\| f\right\| _{n}\geq \lim_{r\to 1^{-}} \frac{(1+r)^{2}}{(1-r)(1+\log^{2}(1-r))}=\infty.
	\end{equation*}
\end{ex}
\begin{rem}
	The harmonic mapping $f=h+\overline{g}$ defined by (\ref{E}) shows that $h$ and $g$ being normal does not imply that $f$ is normal. That example shows that the study of normal harmonic mappings is not reduced to meromorphic case.	
\end{rem}
\begin{rem}
	The function defined in (\ref{E}) also allows us to justify that in general the uniform limit of a sequence of normal harmonic mappings is not a normal harmonic mapping. In effect, if
	\begin{equation*}
	f_{n}(z)=\frac{z}{1-z}-\left( 1-\frac{1}{n}\right) \overline{ \left(  \frac{z}{1-z}+\log (1-z)\right)} ,
	\end{equation*}	
	then $f_{n}$ is normal for all $n$ and converges uniformly to $f$.
\end{rem}
The following result gives us a tool to verify that a given function is not normal. In addition, it gives an idea of the behaviour of a harmonic normal function near the boundary of $\mathbb{D}.$
\begin{prop}
	Let $f$ be a harmonic mapping in $\mathbb{D}$ and let $\left\{  z_{n}\right\}$,	$\left\{  z_{n}^{\prime}\right\}$ be sequences in $\mathbb{D}$ such that $\lim\limits_{n\rightarrow\infty}\rho\left(  z_{n},z_{n}^{\prime}\right)  =0$ and $\lim\limits_{n\rightarrow\infty}\left\vert z_{n}\right\vert =1.$ If
	\[
	\lim_{n\rightarrow\infty}f\left(  z_{n}\right)  =\alpha\text{ \ \ and \ \ }%
	\lim_{n\rightarrow\infty}f\left(  z_{n}^{\prime}\right)  =\beta;\qquad\left(
	\alpha\neq\beta\right)  ,
	\]
	then $f$ is not a normal function.
\end{prop}
\begin{proof}
	Arguing by contradiction, let us assume that
	\[k:=\sup\limits_{z\neq w}\frac{\chi\left(  f\left(  z\right)  ,f\left(  w\right)
		\right)  }{\rho\left(  z,w\right)  }<\infty.\]
	Thus, by hypothesis
	\[
	\chi\left(  f\left(  z_{n}\right)  ,f\left(  z_{n}^{\prime}\right)  \right)
	\leq k\rho \left(  z_{n},z_{n}^{\prime}\right)  \longrightarrow0,\text{\qquad
	}n\longrightarrow\infty.
	\]
	Two cases may occur:
	$\left( i\right)$ If $\ \alpha,\beta\in\mathbb{C},$ for all $n,$
	\[
	0<\chi\left(  \alpha,\beta\right)  \leq\chi\left(  \alpha,f\left(
	z_{n}\right)  \right)  +\chi\left(  f\left(  z_{n}\right)  ,f\left(
	z_{n}^{\prime}\right)  \right)  +\chi\left(  f\left(  z_{n}^{\prime}\right)
	,\beta\right).
	\]	
	Letting $n\longrightarrow\infty$ we obtain $\chi\left(  \alpha,\beta\right)=0.$
	
	$\left(  ii\right)$ Similar argument apply to the case where $\alpha=\infty.$ Indeed, for all $n,$
	\begin{align*}
	0 <\chi\left(  \infty,\beta\right) &\leq\chi\left(  \infty,f\left(z_{n}\right)  \right)  +\chi\left(  f\left(  z_{n}\right),f\left(
	z_{n}^{\prime}\right)  \right)  +\chi\left(  f\left(  z_{n}^{\prime}\right)
	,\beta\right)\\
	& =\frac{1}{\sqrt{1+\left\vert f\left(  z_{n}\right)  \right\vert^{2}}}+\chi\left(  f\left(  z_{n}\right)  ,f\left(  z_{n}^{\prime}\right)
	\right)  +\chi\left(  f\left(  z_{n}^{\prime}\right)  ,\beta\right).
	\end{align*}
	By taking limits as $n\longrightarrow\infty$ we have $\chi\left(  \alpha,\beta\right)=0.$ In both cases a contradiction is obtained.
\end{proof}

\section{Main results }

In the meromorphic case, a function $f$ is normal in $\mathbb{D}$ if and only if the family $\left\{  f\circ\varphi:\varphi\in\mathrm{Aut}\left(
\mathbb{D}\right)  \right\}  $ is normal in $\mathbb{D}$ in the sense of Montel. The following is a partial harmonic analogue. 
\begin{thm}
	Let $f=h+\overline{g}$ be a normal harmonic mapping in $\mathbb{D}$, then $\mathcal{F}:= \left\{f\circ\varphi:\varphi\in\mathrm{Aut}\left(\mathbb{D}\right)\right\}$ is a normal family in $\mathbb{D}.$ 
\end{thm}
\begin{proof}
	Given $\varphi\in\mathrm{Aut}\left(\mathbb{D}\right),$ let $f_{\varphi}:=f\circ\varphi\in\mathcal{F}.$ By Remark\,\ref{rem 1}, $\left\|f_{\varphi} \right\|_n=\left\|f \right\|_n,$ and so there is $M>0$ such that
	\[
	\left(  1-\left\vert z\right\vert ^{2}\right)  \frac{\left\vert \left(
		h\circ\varphi\right)  ^{\prime}\left(  z\right)  \right\vert +\left\vert
		\left(  g\circ\varphi\right)  ^{\prime}\left(  z\right)  \right\vert
	}{1+\left\vert f_{\varphi}\left(  z\right)  \right\vert ^{2}}\leq M,
	\]
	for all $z\in\mathbb{D}.$ Let us fix $z_{1}\in\mathbb{D}$ and let $r=\dfrac
	{1-\left\vert z_{1}\right\vert }{2}$. For all $z\in\overline{\Delta\left(
		z_{1},r\right)  },$ we consider $\gamma\left(  t\right)  =z_1+t\left(
	z-z_{1}\right)  ,$ $t\in\left[  0,1\right]  .$ Thus, if $\Gamma=f_{\varphi
	}\circ\gamma$ we have
	\[
	\chi\left(  f_{\varphi}\left(  z\right)  ,f_{\varphi}\left(  z_{1}\right)
	\right)  \leq \int_{\Gamma} \frac{\left\vert dz\right\vert }{1+\left\vert z\right\vert ^{2}}=%
	\int_{0}^{1} \frac{\left\vert df_{\varphi}\left(  \gamma\left(  t\right)  \right)
		\right\vert \left\vert z-z_{1}\right\vert dt}{1+\left\vert f_{\varphi}\left(
		\gamma\left(  t\right)  \right)  \right\vert ^{2}}\leq\widehat{M}\left\vert
	z-z_{1}\right\vert,
	\]
	where $\widehat{M}$ depends on $M$ and $z_{1}.$ Then, $\mathcal{F}$ is spherically equicontinuous on compact sets in $\mathbb{D}.$ It follows, from the Arzel\'a-Ascoli Theorem, that 
	$\mathcal{F}$ is normal.
\end{proof}
\begin{rem}
	The converse of the previous theorem is not true. Indeed, let us consider the family $\mathcal{F=}\left\{  f\circ\varphi:\varphi\in\mathrm{Aut}\left(
	\mathbb{D}\right)  \right\} $ with $f$ as in (\ref{E}). This family is normal in $\mathbb{D}$.
\end{rem}
Now it is interesting to know which subclasses of harmonic functions turn out to be normal. 
From Theorem 3 in \cite{Co89} it follows that if $f$ is a bounded harmonic mapping in $\mathbb{D}$, then $f$ is normal. Moreover, if $\left\vert f\left(  z\right)  \right\vert \leq k,$ for all $z\in\mathbb{D}%
$, then $\left\Vert f\right\Vert _{n}\leq\dfrac{4k}{\pi}.$ The function
$f\left(  z\right)  =\dfrac{2k}{\pi}\mathrm{Arg}\left(  \dfrac{1+z}{1-z}\right)
$ proves that the inequality is sharp. Also, it follows from
\[
\left(  1-\left\vert z\right\vert ^{2}\right)  \frac{\left\vert h^{\prime
	}\left(  z\right)  \right\vert }{1+\left\vert h\left(  z\right)  \right\vert
	^{2}}\leq\left(  1-\left\vert z\right\vert
^{2}\right)  \frac{\left\vert h^{\prime
	}\left(  z\right) \right\vert + \left\vert g^{\prime
	}\left(  z\right)  \right\vert  }{1+\left\vert f\left(  z\right)
	\right\vert ^{2}}\left(  1+\left\vert f\left(  z\right)  \right\vert
^{2}\right),
\]
that if $f=h+\overline{g}$ is a bounded harmonic mapping in $\mathbb{D},$ then $h$ is normal. In a completely analogous way, we also deduce that $g$ is normal.\\

In the following proposition and in the sequel, $\left\Vert \omega\right\Vert _{\infty}:=\sup\limits_{z\in\mathbb{D}}\left\vert \omega_{f}\left(  z\right)  \right\vert$, where $\omega_{f}=g^{\prime}/h^{\prime}$ is the second complex dilatation of $f$.
\begin{prop}
	Let $f=h+\overline{g}$ be a sense-preserving harmonic mapping in $\mathbb{D}$ with $\left\| \omega\right\|_{\infty}<1$. If $h$ is a starlike function, then $f$ is normal.
\end{prop}
\begin{proof}
	Let $z\in\mathbb{D}$ and $\Gamma=h^{-1}(\gamma)$, where $\gamma$ is the segment from $0$ to $h(z)$. Thus
	\begin{equation*}
	\left| h(z)\right| =\int_{\gamma}\left| dw\right|= \int_{\Gamma}\left| h'(\zeta)\right|\left| d\zeta\right|\geq \frac{1}{ \left\| \omega\right\|_{\infty} }\int_{\Gamma}\left| g'(\zeta)\right|\left| d\zeta\right|\geq \frac{1}{ \left\| \omega\right\|_{\infty} }\left| g(z)\right|.
	\end{equation*}
	It follows that $\dfrac{\left| g(z)\right|}{\left| h(z)\right|}\leq \left\| \omega\right\|_{\infty}$ for all $z\in\mathbb{D}$. On the other hand, since $h$ is normal, there is $M>0$ such that
	\begin{align*}
	\left(  1-\left\vert z\right\vert^{2}\right)  \frac{\left\vert h^{\prime}\left(  z\right) \right\vert + \left\vert g^{\prime}\left(  z\right)  \right\vert  }{1+\left\vert f\left(  z\right)\right\vert ^{2}}
	&\leq \left(  1-\left\vert z\right\vert ^{2}\right)  \frac{2\left\vert h^{\prime}\left(  z\right)  \right\vert }{1+\left\vert h\left(  z\right)  \right\vert^{2}}\frac{1+\left\vert h\left(  z\right)  \right\vert	^{2}}{1+\left\vert f\left(  z\right)  \right\vert^{2}}\\
	&\leq M\frac{1+\left\vert h\left(  z\right)  \right\vert^{2}}{1+\left\vert f\left(  z\right)  \right\vert^{2}}\\
	&\leq M\left( 1+\frac{\left\vert h\left(  z\right)  \right\vert^{2}}{\left\vert f\left(  z\right)  \right\vert^{2}}\right) 
	\end{align*}
	for all $z\in\mathbb{D}.$ Hence,
	\begin{align*}
	\left(  1-\left\vert z\right\vert^{2}\right)  \frac{\left\vert h^{\prime}\left(  z\right) \right\vert + \left\vert g^{\prime}\left(  z\right)  \right\vert  }{1+\left\vert f\left(  z\right)\right\vert ^{2}}
	&\leq M\left( 1+\left( 1-\frac{\left| g(z)\right|}{\left| h(z)\right|}\right) ^{-2}\right)\\
	&\leq M\left( 1+\left( 1-\left\| \omega\right\|_{\infty}\right) ^{-2}\right)
	\end{align*}
	for all $z\in\mathbb{D},$ which completes the proof.
\end{proof}
It is known that all univalent meromorphic functions are normal, however the function considered in Example\,\ref{ex 2} shows that this is not true in the harmonic case. In the following theorem we prove that a univalent harmonic function $f$ is normal if satisfies the additional condition $\left\| \omega\right\|_{\infty}<1.$
\begin{thm}
	Let $f=h+\overline{g}$ be a univalent harmonic mapping in $\mathbb{D}$ with $\left\|\omega\right\|_{\infty}<1$. Then $f$ is normal. The condition  $\left\Vert\omega\right\Vert _{\infty}<1$ can in general not be omitted.
\end{thm}
\begin{proof}
	Let $F\in S_{H}$ such that $f=aF+b=aH+b+\overline{\overline{a}G},$ $a,b\in\mathbb{C}$ constants, $a\neq 0$. Thus, for all $z\in\mathbb{D},$
	\begin{equation}\label{G}
	\begin{split}
	\frac{\left(  1-\left\vert z\right\vert ^{2}\right)  \left(  \left\vert
		h^{\prime}\left(  z\right)  \right\vert +\left\vert g^{\prime}\left(
		z\right)  \right\vert \right)  }{1+\left\vert f\left(  z\right)  \right\vert
		^{2}} &  =\frac{\left(  1-\left\vert z\right\vert ^{2}\right)  \left(
		\left\vert aH^{\prime}\left(  z\right)  \right\vert +\left\vert \overline
		{a}G^{\prime}\left(  z\right)  \right\vert \right)  }{1+\left\vert aF\left(
		z\right)  +b\right\vert ^{2}}\\
	&  \leq\frac{\left(  1-\left\vert z\right\vert ^{2}\right)  \left\vert
		aF\left(  z\right)  \right\vert }{1+\left\vert aF\left(  z\right)
		+b\right\vert ^{2}}\,\frac{2\left\vert H^{\prime}\left(  z\right)  \right\vert
	}{\left\vert F\left(  z\right)  \right\vert }.
	\end{split}
	\end{equation}
	On the other hand, given $z\in\mathbb{D}$ fixed, we consider the Koebe transform 
	\[
	k\left(  \zeta\right)  =\frac{F\left(  \frac{z+\zeta}{1+\overline{z}\zeta
		}\right)  -F\left(  z\right)  }{\left(  1-\left\vert z\right\vert ^{2}\right)
		H^{\prime}\left(  z\right)  }.
	\]
	Notice that
	\[
	\omega_{F}=\frac{a}{\overline{a}}\,\omega_{f}\text{ \ \ \ \ \ and \ \ \ \ \ }%
	\omega_{k}=\omega_{F}\left(  \frac{z+\zeta}{1+\overline{z}\zeta}\right)
	=\frac{a}{\overline{a}}\,\omega_{f}\left(  \frac{z+\zeta}{1+\overline{z}\zeta
	}\right)  .
	\]
	Also, $k\in S_{H},$
	\[b_{1}=b_{1}\left(  k\right)  =w_{k}\left(0\right) =\frac{a}{\overline{a}}\,\omega_{f}\left(z\right)\quad\text{and}\quad k_{0}=\dfrac{k-\overline{b_{1}}\overline{k}}{1-\left\vert b_{1}\right\vert
		^{2}}\in S_{H}^{0}.\]
	Thus, by Theorem\,1 in \cite{SM90}, we have
	\[
	\left\vert k_{0}\left(  \zeta\right)  \right\vert \geq\frac{1}{2\alpha}\left[
	1-\left(  \frac{1-\left\vert \zeta\right\vert }{1+\left\vert \zeta\right\vert
	}\right)  ^{\alpha}\right],
	\]
	where $\alpha$ is the supremum of $|a_2|$ among all functions in $S_H$. With $\zeta=-z,$ we obtain
	\[
	\left\vert \frac{k\left(  -z\right)  -\overline{b_{1}}\overline{k\left(
			-z\right)  }}{1-\left\vert b_{1}\right\vert ^{2}}\right\vert \geq\frac
	{1}{2\alpha}\left[  1-\left(  \frac{1-\left\vert z\right\vert }{1+\left\vert
		z\right\vert }\right)  ^{\alpha}\right]  ,
	\]
	whence
	\[
	\frac{\left\vert -F\left(  z\right)  +\overline{b_{1}}\;\overline{F\left(
			z\right)  }\right\vert }{\left(  1-\left\vert b_{1}\right\vert ^{2}\right)
		\left(  1-\left\vert z\right\vert ^{2}\right)  \left\vert H^{\prime}\left(
		z\right)  \right\vert }\geq\frac{1}{2\alpha}\left[  1-\left(  \frac
	{1-\left\vert z\right\vert }{1+\left\vert z\right\vert }\right)  ^{\alpha
	}\right]  .
	\]
	From this, there is $C>0$ such that
	\[
	\frac{\left(  1-\left\vert z\right\vert ^{2}\right)  \left\vert H^{\prime
		}\left(  z\right)  \right\vert }{\left\vert F\left(  z\right)  \right\vert
	}\,\frac{\left(  1-\left\vert b_{1}\right\vert ^{2}\right)  }{1+\left\vert
		b_{1}\right\vert }\leq\frac{2\alpha}{1-\left(  \frac{1-\left\vert z\right\vert
		}{1+\left\vert z\right\vert }\right)  ^{\alpha}}\leq C,
	\]
	for every $z$ with $1/2\leq\left\vert z\right\vert <1.$ Therefore, for those $z$ we have
	\begin{equation}\label{H}
	\frac{\left(  1-\left\vert z\right\vert ^{2}\right)  \left\vert H^{\prime
		}\left(  z\right)  \right\vert }{\left\vert F\left(  z\right)  \right\vert
	}\leq\frac{C}{1-\left\vert b_{1}\right\vert }=\frac{C}{1-\left\vert \omega
		_{f}\left(  z\right)  \right\vert }\leq\frac{C}{1-\left\Vert \omega\right\Vert
		_{\infty}}.
	\end{equation}
It follows from (\ref{G}) and (\ref{H}) that $f$ is normal.
	
That the condition  $\left\Vert \omega \right\Vert _{\infty}<1$ cannot be omitted is justified by the function $f$ defined in (\ref{E}). 
\end{proof}
In our last theorem we will show an integral condition which implies that $f$ is normal. This result arises looking for an analogous to Pommerenke's integral criterion given by condition (\ref{DI1Po}). Namely, a meromorphic function is normal if satisfies the condition
\begin{equation*}  
\iint_\mathbb{D} {\left\lbrace f^{\#}(z)\right\rbrace }^2dA< \infty.
\end{equation*}
Given a nonnegative subharmonic function $u$ defined in $\mathbb{D}$ let us define 
\[
N(r,u)=\int_{0}^{r} \frac {\mu (B(0,t))}{t}dt,
\]
where $\mu$ is the Riesz measure of $u$. Namely, if $u \in C^{2}(\mathbb{D})$ then $d\mu=\Delta u\,dx,$ where $dx$ is the Lebesgue measure (see \cite{HK76}).

For all $0<r<1,$ $u\in C^{2}(\overline{D(0,r)}).$ Thus, $u$ restricted to the disk $D(0,1/2)$ can be extended to a subharmonic function in all $\mathbb{R} ^{2}$ (see \cite{VMP01}, Theorem 5.1). Then, by Theorem 1.2 in \cite{KZ03}, there are constants $\delta >0$ and $c>0$ such that 
\begin{equation}\label{aux subharmonic}
M(r,u)\leq 2u(0)+cN(r\delta,u)
\end{equation}
for all  nonnegative subharmonic function $u\in C^{2}(\overline{B(0,1/2)})$ and all $r\in (0,1)$ with $r\delta<1/2$, where 
\[ 
M(r,u) = \sup \left\lbrace u(z) : |z|\leq r \right\rbrace =\max \left\lbrace u(z) : |z|=r \right\rbrace .
\] 
\begin{thm}
	Let $f=h+\bar{g}$ be a harmonic mapping in $\mathbb{D}$ which omits one value in $\mathbb{C}.$ Suppose that there exists $\alpha \in (0,1)$ satisfying 
	\[
	\frac{1}{l(\partial\Omega)}\int_{\partial\Omega} (|h'(\zeta)|+|g'(\zeta)|)\,|d\zeta|\leq \frac{1}{r^{\alpha}},
	\]
	for all $r\in (0,1)$ and for all $\sigma\in \mathrm{Aut}(\mathbb{D})$ with $\Omega= \sigma(D(0,r))$. Then $f$ is normal.
\end{thm}
\begin{proof}
	Arguing by contradiction, suppose there are $ \left\lbrace z_{n}\right\rbrace ,\, \left\lbrace w_{n}\right\rbrace \subset \mathbb{D} $ such that 
	\begin{equation}\label{I}
	I_{n}=\frac{\chi\left(  f\left(  z_n\right)  ,f\left(  w_n\right)
		\right)  }{\rho \left(  z_n,w_n\right)  }\geq n,
	\end{equation}
	for all $n\in \mathbb{N}.$
	
	Now, given $n,$ we define
	\[\sigma_{n}(z)=\frac{z_n-z}{1-\overline{z_{n}}z}\,,\qquad \sigma_{n}(p_{n})=w_{n},\qquad \text{and}\qquad f_{n}=f\circ\sigma_{n}.\]
	Thus,
	\[f_{n}(0)=f(z_{n}),\qquad f_{n}(p_{n})=f(w_{n}),\qquad\text{and}\qquad \rho(0,p_{n})=\rho(z_{n},w_{n}).\]
	From this and (\ref{I}), we get $\rho(0,p_{n})\to 0$ and therefore $p_{n}\to 0.$ Hence, there is no loss of generality if we assume that $|p_{n}|<1/4$ for all $n\in \mathbb{N}$. Note that 
	\begin{equation*}
	I_{n}=\frac{\chi\left(  f_{n}\left( 0\right)  ,f_{n}\left( p_{n}\right)
		\right)  }{\rho\left(  0,p_{n}\right)  }\lesssim \frac{|f_{n}(0)-f_{n}(p_{n})|}{|p_{n}|}\frac{1}{\sqrt{1+|f_{n}(0)|}}.
	\end{equation*} 
	On the other hand, if $\varphi:\mathbb{D}\to \mathbb{R}$ is harmonic, given $0<\rho<1,$ 
	\begin{equation*}
	\varphi_{i}(z)=\frac{1}{2\pi}\int_{0}^{2\pi}P_{i}(z,\theta)\varphi(\rho e^{i\theta})d\theta, \qquad i=1,2
	\end{equation*}
	for all $|z|<\rho,$ which gives
	\begin{equation*}
	|\varphi_{i}(z)|\leq\frac{1}{2\pi}\int_{0}^{2\pi}C_{\rho}M(\rho,\varphi)d\theta,
	\end{equation*}
	for all $|z|<\rho/2,$ where
	\[M(\rho,\varphi)=\max \left\lbrace|\varphi(w)|:|w|=\rho\right\rbrace\qquad\text{and}\qquad |P_{i}(w)|\leq C_{\rho},\]
	$i\in \{1,2\}$ and $|w|<\rho/2.$
	From here and (\ref{I}), for  $\rho\in \left( 0,\frac{1}{2}\right] $
	\begin{equation}\label{J}
	I_{n}\lesssim \frac{C_{\rho}M(\rho,f_n)}{\sqrt{1+|f_{n}(0)|^{2}}}, 
	\end{equation}
	for all $n\in \mathbb{N}$ and $|p_{n}|\leq\rho/2$.
	
	On the other hand, as $f_{n}\vert_{\overline{B(0,\frac{1}{2})}}$ has a subharmonic extension to $\mathbb{R}^{2}$, in virtue of \eqref{aux subharmonic}, there are $c>0$ and $\delta>0$, which do not depend on $f_{n}$, such that  
	\begin{equation}\label{K}
	M(r,|f_{n}|)\leq2|f_{n}(0)|+cN(\delta r,|f_{n}|), 
	\end{equation} 
	for all $r>0$. Let $0<R<\frac{1}{2}$ be such that $\delta R<\frac{1}{2}$. Since  $f_{n}\in C^{2}\left(\overline{B(0,\frac{1}{2})}\right),$ it follows from Green's Theorem that for all $0<t<\delta r,$
	\[\mu (B(0,t))= \int_{B(0,t)} \Delta |f_{n}|(z)dz=\int_{\partial B(0,t)} \left \langle \nabla |f_{n}(\zeta)|, \eta(\zeta)\right\rangle |d\zeta|,\]
	where $\eta$ is the unit normal vector to $\partial B(0,t).$ Hence, if
	\[f_n=h_n+\overline{g_n},\qquad \lambda_n=|h'_n|+|g'_n|,\qquad\text{and}\qquad \lambda=|h'|+|g'|,\]
	we deduce that
	\begin{align*}
	\mu (B(0,t))&\leq \int_{\partial B(0,t)} \lambda_n(\zeta)|d\zeta|\\
	&=\int_{\partial B(0,t)} \lambda(\sigma_{n}(\zeta))|\sigma_{n} ' (\zeta)||d\zeta|\\
	&= \int_{\partial \sigma_{n}(B(0,t))} \lambda(\zeta)|d\zeta|,
	\end{align*}
	for $0<t\leq \delta r<R$. From here and our hypothesis, for all $r\in(0,1)$ with $r\delta <R,$ we have
	\begin{equation*}
	N(\delta r,|f_{n}|)\leq \int_{0}^{\delta r}\frac{2\pi t_{n}}{t}\,\frac{1}{t^{\alpha}}\,dt \lesssim (\delta r)^{1-\alpha},
	\end{equation*}
	where $t_{n}$ is the Euclidean radius of  $\partial \sigma_{n}(B(0,t))$. We obtain, by (\ref {J}) and (\ref {K}), that for all $n$ such that $|p_{n}|<R/2$, 
	\begin{equation*}
	I_{n}\lesssim \frac{C_{R}\, M(R,|f_{n}|)}{\sqrt{1+|f_{n}(0)|^{2}}}\lesssim C_{R}\left[ \frac{|f_{n}(0)|}{\sqrt{1+|f_{n}(0)|^{2}}}+N(\delta r,|f_{n}|)\right] \leq C_{R}\left( 1+(\delta r)^{1-\alpha}\right), 
	\end{equation*}
	which contradicts the assumption in  (\ref{I}).
\end{proof}

\end{document}